\documentclass[12pt]{article}

\newtheorem{definition}{Definition}

\newtheorem{proposition}{Proposition}

\newtheorem{conjecture}{Conjecture}

\usepackage{epsfig}
\usepackage{url}
\usepackage{amssymb}
\usepackage{amsmath}
\usepackage{amsfonts}

\def\Dbar{\leavevmode\lower.6ex\hbox to 0pt{\hskip-.23ex \accent"16\hss}D}

\def\bZ{{\mbox{\bf Z}}}

\newcommand{\nc}{\newcommand}
\nc{\cP}{{\cal P}}

\begin{document}

{\bf\LARGE
\begin{center}
A class of cyclic $(v;k_1,k_2,k_3;\lambda)$ difference
families with $v \equiv 3 \pmod{4}$ a prime
\end{center}
}

{\Large
\begin{center}
Dragomir {\v{Z}}. {\Dbar}okovi{\'c}\footnote{University of Waterloo,
Department of Pure Mathematics, Waterloo, Ontario, N2L 3G1, Canada
e-mail: \url{djokovic@math.uwaterloo.ca}}, Ilias S.
Kotsireas\footnote{Wilfrid Laurier University, Department of Physics
\& Computer Science, Waterloo, Ontario, N2L 3C5, Canada, e-mail:
\url{ikotsire@wlu.ca}}
\end{center}
}

\begin{abstract}
We construct several new cyclic $(v;k_1,k_2,k_3;\lambda)$ difference families, with $v\equiv3\pmod{4}$ a prime and
$\lambda=k_1+k_2+k_3-(3v-1)/4$. Such families can be used in conjunction with the well-known Paley-Todd difference sets to construct skew-Hadamard matrices of order $4v$.
Our main result is that we have constructed for the first time
the examples of skew Hadamard matrices of orders
$4\cdot239=956$ and $4\cdot331=1324$.
\end{abstract}

\section{Introduction}

Let $\bZ_v = \{ 0, 1, \ldots, v-1\}$ be the ring of integers modulo an integer $v>1$. Let $k_1,\ldots,k_t$ be nonnegative
integers, $\lambda$ an integer such that
\begin{equation} \label{par-Lambda}
\lambda (v - 1) = \sum_{i=1}^t k_i (k_i - 1),
\end{equation}
and let $X_1,\ldots,X_t$ be subsets of $\bZ_v$ such that
$|X_i|=k_i$, $i\in\{1,2,\ldots,t\}$.

\begin{definition}
We say that $X_1,\ldots,X_t$ are {\em supplementary difference
sets (SDS)} or a {\em difference family} with parameters
$(v;k_1,\ldots,k_t;\lambda)$, if for every
$c\in\bZ_v\setminus\{0\}$ there are exactly $\lambda$ ordered
triples
$(a,b,i)\in\bZ_v\times\bZ_v\times\{1,2,\ldots,t\}$
such that $\{a,b\}\subseteq X_i$ and $a-b \equiv c \pmod{v}$.
\end{definition}

In the context of an SDS $(v;k_1,\ldots,k_t;\lambda)$, it is
convenient to introduce an additional parameter, {\em order}
$n$, defined by
\begin{equation} \label{par-n}
n = k_1 + \cdots + k_t - \lambda.
\end{equation}
We refer to the sets $X_i$ as the {\em base blocks} of the SDS. In the case $t=1$ the SDSs are called {\em cyclic difference sets}, see \cite{Baumert:1971}, \cite{Handbook:DifferenceSets}, \cite{Stinson:book:2004}.
In some cases we have to replace some of the base blocks
with their complements in $\bZ_v$. This gives a new SDS
with different parameter $\lambda$. However, the parameter
$n$ remains the same. Thus we can easily compute the new
value of $\lambda$.

There are exactly twelve integers $n<500$ for which the existence question for Hadamard matrices of order $4n$ is still undecided  (see \cite{DGK:JCD:2014}):
$$
167,179,223,283,311, 347,359,419,443,479, 487,491.
$$
All of these integers are primes congruent to 3 modulo 4.
This motivates the search for special methods to
construct Hadamard matrices for these orders.

One of the powerful methods is based on the Goethals-Seidel array (GS-array), see e.g. \cite{SY:1992}:
$$
\left[ \begin{array}{cccc}
Z_0 & Z_1R & Z_2R & Z_3R \\
-Z_1R & Z_0 & -Z_3^T R & Z_2^T R \\
-Z_2R & Z_3^T R & Z_0 & -Z_1^T R \\
-Z_3R & -Z_2^T R & Z_1^T R & Z_0
\end{array} \right],
$$
where the $Z_i$ are suitable $\{1,-1\}$-circulant matrices and
$R$ is the square matrix with ones on the back-diagonal and zeros
elsewhere.

If $X\subseteq\bZ_v$, then the {\em associated sequence} of $X$
is the $\{\pm1\}$-sequence $\xi_0,\xi_1,\ldots,\xi_{v-1}$
with $\xi_i=-1$ if $i\in X$ and $\xi_i=1$ otherwise.

Usually, to construct a Hadamard matrix of order $4v$ via the
GS-array, we require an SDS, say $X_0,X_1,X_2,X_3$, with parameters $(v;k_0,k_1,k_2,k_3;\lambda_0)$ and order $n_0=v$,
i.e., such that $\lambda_0=k_0+k_1+k_2+k_3-v$. Given such
SDS, let $A_0,A_1,A_2,A_3$ be the sequences associated to the
base blocks $X_0,X_1,X_2,X_3$, and let $Z_0,Z_1,Z_2,Z_3$
be the circulant matrices with the first rows
$A_0,A_1,A_2,A_3$, respectively.
By plugging these circulants into the GS-array, we obtain a
Hadamard matrix. This will be a skew-Hadamard matrix if the block
$X_0$ is of skew type. This means that $X_0$ has the property
that $i\in X_0$ if and only if $-i\notin X_0$. In particular,
$v$ must be odd, $0\notin X_0$ and $|X_0|=(v-1)/2$.

In this paper we are interested in the special case of this construction where $v \equiv 3 \pmod{4}$ is a prime and
$k_0=(v-1)/2$.
Moreover, except for the last section, we shall assume that we have choosen $X_0$ to be the Paley-Todd difference set, see
\cite[Theorem 27.4, p. 234]{vanLint:Wilson:1992}. Then the base blocks $X_1,X_2,X_3$ form an SDS with the parameter set
$(v;k_1,k_2,k_3;\lambda)$ and order $n=(3v-1)/4$. Thus we have
\begin{equation} \label{par-lambda}
\lambda = k_1 + k_2 + k_3 - \frac{3v-1}{4}.
\end{equation}
Without any loss of generality, we may assume that
\begin{equation}
v/2>k_1\ge k_2\ge k_3\ge 0.
\label{nejednakosti}
\end{equation}

We denote by $\cP$ the collection of parameter sets
$(v;k_1,k_2,k_3;\lambda)$, with
$v \equiv 3 \pmod{4}$ a prime, and satisfying the
conditions (\ref{par-Lambda}), (\ref{par-lambda}) and
(\ref{nejednakosti}).

\begin{conjecture} \label{hip} For each parameter
set in $\cP$, there exists at least one SDS.
\end{conjecture}

We claim that the family $\cP$ is infinite. This follows
from the following stronger result.

\begin{proposition}
\label{beskonacna}
For each prime number $v \equiv 3 \pmod{4}$, there exist
nonnegative integers $k_1,k_2,k_3$ and $\lambda$ such that
$(v;k_1,k_2,k_3;\lambda)\in\cP$.
\end{proposition}
\begin{proof}
Recall the famous result of Gauss that every positive
integer is a sum of at most three triangular numbers.
This fact is equivalent to the assertion that every
positive integer congruent to 3 modulo 8 is a sum of
three odd squares (see \cite{Duke}).

As $(3;1,1,0;0)\in\cP$, we may assume that $v>3$.
Since $v$ is odd, we have $4v-1\equiv 3 \pmod{8}$. Hence,
there exist positive odd integers $s_1,s_2,s_3$
such that $\sum s_i^2=4v-1$. Note that
$s_i < \sqrt{4v-1} < v$, and so the integers
$k_i=(v-s_i)/2$ are positive and less than $v/2$.
By permuting the $k_i$, we may assume that $k_1\ge k_2\ge k_3$.
It is now easy to verify that $(v;k_1,k_2,k_3;\lambda)\in\cP$,
where $\lambda=\sum k_i -(3v-1)/4$.
\end{proof}

In section \ref{family} (and in the appendix) we list all
parameter sets in $\cP$ with $v\le131$ and for each of them we provide at least one SDS whenever such SDS is known. Although there remain several undecided cases, these computational
results suggest that Conjecture \ref{hip} is true.

In section \ref{SkewHM239} we give three non-equivalent SDSs with the parameter set $(239;119,112,106;158)$ and $n=179$. The importance of these three SDSs is that each of them, in conjunction with the Paley-Todd difference set (the set of nonzero squares in the finite field $\bZ_{239}$), gives a skew-Hadamard matrix of size $4\cdot239$.

Finally, for the parameter set $(331;165,155,155,155;299)$,
which is not in $\cP$, we construct in section \ref{SkewHM331}
six non-equivalent SDSs in which the block of size $165$ is
of skew type. Thus we obtain six skew-Hadamard matrices of
order $4\cdot331=1324$.

To the best of our knowledge, skew-Hadamard matrices of the two orders mentioned above were not known previously.
See \cite[Table 1]{KouStyl:2008}) for the list of 98 odd
positive integers $m<500$ for which no skew-Hadamard matrix of
order $4m$ was known at that time. Subsequently, skew-Hadamard
matrices of order $4m$ were constructed for $m=109,145,247$
\cite{Djokovic:JCD:2008} and $m=213$ \cite{DGK:JCD:2014}.
To update the above mentioned list, one should delete from it
the six integers $109,145,213,239,247,331$. The first 22 entries
(those with $4m<1000$) in the updated list are:
\begin{eqnarray*}
&& 69,89,101,107,119, 149,153,167,177,179, 191,193, \\
&& 201,205,209, 223,225,229,233,235, 245,249.
\end{eqnarray*}

Only five of these values, namely $v=107,167,179,191,223$,
admit the parameter sets in $\cP$.

\section{Difference families with parameters in $\cP$ and
$v\le131$} \label{family}

In this section we provide evidence for Conjecture \ref{hip}.
In the appendix we list all parameter sets
$(v;k_1,k_2,k_3;\lambda)$ with $v\le 131$ which belong to
the family $\cP$.
For each of them we either give a reference to papers where
the SDSs have been constructed or give explicit examples of
SDSs that we have constructed.
If no SDSs are known, we indicate by a question
mark that the existence question remains undecided.

In some cases we use the known D-optimal SDSs to construct
the desired difference family with three base blocks.
This works only when $k_1=(v-1)/2$. Assume that there exists
an SDS $(X_2,X_3)$ with the parameter set
$(v;k_2,k_3;\lambda')$, where $\lambda'=k_2+k_3-(v-1)/2$.
If $X_1$ is the Paley-Todd diference set in $\bZ_v$, then
$(X_1,X_2,X_3)$ is an SDS with parameters
$(v;k_1,k_2,k_3;\lambda)\in\cP$.
In that case we say that the SDS $(X_1,X_2,X_3)$ is constructed
from the D-optimal SDS $(X_2,X_3)$. A list of known D-optimal
SDSs $(v;k_2,k_3;\lambda')$ with $v<100$ and examples of the
corresponding DO-designs are given in \cite{DK:D-optimal:2015}.

The multiplicative group $\bZ_v^*$ of the prime field $\bZ_v$,
acts on $\bZ_v$ by multiplication modulo $v$. In many cases we construct the base blocks $X_1,X_2,X_3$ of an SDS as the union of orbits of a nontrivial subgroup $H$ of $\bZ_v^*$.
Let $H=\langle h \rangle$, i.e., $h$ is a generator of $H$.
The orbit containing the integer $j\in\bZ_v$ is written as
$H\cdot j$. In particular, when $j=0$ we obtain the trivial orbit
$H\cdot 0=\{0\}$.
In all of our computations, the order of $H$ is a small odd prime
divisor $q$ of $v-1$. So, an $H$-orbit has size 1 or $q$.
If it is of size 1, say $\{x\}$, then $(h-1)x=0$.
As $v$ is a prime, it follows that $x=0$.
Hence, all nontrivial orbits of $H$ are of size $q$.

We write the blocks $X_1,X_2,X_3$ as
\begin{equation} \label{orbite}
X_1=\bigcup_{j\in J} H\cdot j, \quad
X_2=\bigcup_{k\in K} H\cdot k, \quad
X_3=\bigcup_{l\in L} H\cdot l.
\end{equation}
Here, the set $J$ is a set of representatives of the $H$-orbits
comprising $X_1$, etc.
Thus, in order to specify the blocks $X_1,X_2,X_3$, it suffices
to specify the subgroup $H$ and the corresponding sets of
representatives $J,K,L\subseteq\bZ_v$.

In order to be able to use this method it is necessary that
$H$ be chosen so that $q$ divides either $k_i$ or $v-k_i$
for each $i$. This explains why we were able to find an SDS
for $v=239$ but not for some smaller values of $v$,
say for $v=107$.

Table 1 summarizes what we presently know about the
existence of SDSs with parameters in $\cP$ for $v\le131$.
The entry ``yes'' means that the SDS exists, while ``?'' means
that the existence question remains undecided.

$$ \begin{array}{cccccc|c|cccccc} \hline
v  &  k_1  &  k_2  &  k_3  &  \lambda  &  \mbox{yes/?} & & v  &  k_1  &  k_2  &  k_3  &  \lambda  &  \mbox{yes/?} \\
\hline \hline
3 &  1 &  1 &  0 &  0 & \mbox{yes} & & 71 &  34 &  32 &  28 &  41 & \mbox{?} \\ \hline
7 &  3 &  3 &  1 &  2 & \mbox{yes} & & 71 &  31 &  31 &  30 &  39 & \mbox{yes} \\ \hline
7 &  2 &  2 &  2 &  1 & \mbox{yes} & & 79 &  39 &  37 &  31 &  48 & \mbox{yes} \\ \hline
11 &  4 &  4 &  3 &  3 & \mbox{yes} & & 79 &  38 &  35 &  32 &  46 & \mbox{?} \\ \hline
19 &  9 &  7 &  6 &  8 & \mbox{yes} & & 79 &  37 &  34 &  33 &  45 & \mbox{yes} \\ \hline
19 &  7 &  7 &  7 &  7 & \mbox{yes} & & 83 &  39 &  37 &  34 &  48 & \mbox{?} \\ \hline
23 &  11 &  10 &  7 &  11 & \mbox{yes} & & 83 &  37 &  37 &  35 &  47 & \mbox{?} \\ \hline
31 &  15 &  15 &  10 &  17 & \mbox{yes} & & 103 &  51 &  48 &  42 &  64 & \mbox{yes} \\ \hline
31 &  13 &  12 &  12 &  14 & \mbox{yes} & & 103 &  51 &  46 &  43 &  63 & \mbox{yes} \\ \hline
43 &  21 &  21 &  15 &  25 & \mbox{yes} & & 103 &  49 &  49 &  42 &  63 & \mbox{yes} \\ \hline
43 &  21 &  18 &  16 &  23 & \mbox{yes} & & 103 &  46 &  46 &  45 &  60 & \mbox{yes} \\ \hline
43 &  20 &  17 &  17 &  22 & \mbox{yes} & & 107 &  49 &  48 &  46 &  63 & \mbox{?} \\ \hline
43 &  19 &  19 &  16 &  22 & \mbox{yes} & & 127 &  61 &  58 &  54 &  78 & \mbox{?} \\ \hline
47 &  22 &  22 &  17 &  26 & \mbox{yes} & & 127 &  60 &  57 &  56 &  77 & \mbox{?} \\ \hline
47 &  21 &  19 &  19 &  24 & \mbox{yes} & & 127 &  57 &  57 &  57 &  76 & \mbox{yes} \\ \hline
59 &  29 &  28 &  22 &  35 & \mbox{yes} & & 131 &  65 &  61 &  55 &  83 & \mbox{yes} \\ \hline
67 &  31 &  28 &  28 &  37 & \mbox{yes} & & 131 &  64 &  58 &  57 &  81 & \mbox{?} \\ \hline
67 &  30 &  30 &  27 &  37 & \mbox{yes} & & 131 &  61 &  61 &  56 &  80 & \mbox{yes} \\ \hline
\end{array} $$
\begin{center} Table 1
\end{center}

\section{ Three skew-Hadamard matrices of order $4\cdot239$
\label{SkewHM239} }

According to the list \cite[Table 1]{KouStyl:2008} no skew-Hadamard matrix of order $4\cdot239$ is known. We construct such a matrix by using the method presented in the Introduction.

For this construction we need a difference family with three base blocks $X_1,X_2,X_3$ having the parameters
$(239;119,112,106;158)$. We have constructed three such families by using the subgroup $H=\{1,10,24,44,98,100,201\}$ of
$\bZ_{239}^*$. In each case, the base blocks $X_1,X_2,X_3$ are unions of orbits of $H$ as in (\ref{orbite}).
We have constructed such three non-equivalent SDSs. Their sets $J,K,L$ are:
\begin{eqnarray*}
J_1 &=& \{1,3,5,6,15, 17,19,28,34,38, 39,57,58,63,85, 95,107\}, \\
K_1 &=& \{1,3,4,5,15, 16,17,18,19,21, 23,29,35,45,58, 63\},  \\
L_1 &=& \{0, 1,4,6,7,8,  13,16,18,34,35, 45,47,58,63,95\};
\end{eqnarray*}

\begin{eqnarray*}
J_2 &=& \{1,3,9,13,14, 15,16,18,23,28, 29,38,42,45,58, 85,107\}, \\
K_2 &=& \{4,6,7,8,9,   13,14,17,28,39, 45,47,57,58,95, 107 \},  \\
L_2 &=& \{0, 1,4,5,6,9,   13,18,21,34,35, 39,42,57,85,95  \};
\end{eqnarray*}

\begin{eqnarray*}
J_3 &=& \{1,2,3,4,5,    6,7,8,14,16,    18,19,34,39,57, 58,95\}, \\
K_3 &=& \{5,9,14,15,17, 18,21,23,29,34, 35,39,45,47,57, 58\},  \\
L_3 &=& \{0, 1,6,7,8,13,  15,16,17,23,28, 35,57,63,85,107\},
\end{eqnarray*}
respectively.

\section{ Six skew-Hadamard matrices of order $4\cdot331$
\label{SkewHM331} }

According to the list \cite[Table 1]{KouStyl:2008} no
skew-Hadamard matrix of order $4\cdot331$ is known.
Although there are six parameters sets in $\cP$ with
$v=331$, we did not succeed to construct an SDS for
any of them. However we were able to construct six
difference families with base blocks $X_0,X_1,X_2,X_3$ having the parameter set $(331;165,155,155,155;299)$.
In all six cases $X_0$ is of skew type (but not a difference set). Hence, by plugging the corresponding four circulant matrices into the GS-array, we obtain a skew-Hadamard matrix of order $4\cdot331=1324$.
We have verified that the six difference families constructed
above are pairwise non-equivalent.

We have constructed these 6 families by using the subgroup
$$
H=\{1,74,80,85,111,120,167,180,270,274,293\}
$$
of $\bZ_{331}^*$ of order 11.
All four base blocks are unions of orbits of $H$ as in
(\ref{orbite}).

For the first two families we use the sets of orbit representatives $M,J,K,L$ and $M,J',K,L$, where
\begin{eqnarray*}
M &=& \{2,4,8,10,14, 16,20,28,38,31, 32,37,56,62,64\}, \\
J &=& \{0, 1,2,5,10,13, 19,22,28,31,37, 53,56,64,101\}, \\
J'&=& \{0,7,8,13,14, 19,28,31,37,38, 49,53,62,73,76\}, \\
K &=& \{0, 1,5,7,8,10, 11,16,31,38,41, 56,62,64,73\},  \\
L &=& \{0, 1,7,8,10,13, 14,19,22,32,37, 44,62,73,76\}.
\end{eqnarray*}

For the remaining four families we use the sets
$M,J,K,L$;~ $M,J,K,L'$;~ $M,J,K',L$,~ and $M,J,K',L'$, where
\begin{eqnarray*}
M &=& \{4,13,14,16,22, 32,37,38,41,49, 53,56,62,64,76\}, \\
J &=& \{0, 2,10,11,20,22, 31,32,37,38,53, 62,64,76,101\}, \\
K &=& \{0, 2,10,11,20,22,31,32,37,38,53,62,64,76,101\}, \\
K'&=& \{0, 1,2,4,5,7, 8,13,14,19,28, 41,49,76,88\}, \\
L &=& \{0, 4,8,11,13,16, 19,20,22,31,37, 38,49,64,101\}, \\
L'&=& \{0, 1,4,7,8,10,11,13,14,19,22,31,37,41,44\}.
\end{eqnarray*}
(In the last four cases the block $X_1$ is symmetric.)

\section{Conclusions}

One of the standard methods of construction of
Hadamard matrices of order $4v$ uses the Goethals--Seidel array.
This method requires supplementary difference sets (SDSs)
with 4 basic blocks and special parameter sets, namely
$(v;k_0,k_1,k_2,k_3;\lambda_0)$ with order $n_0=v$.

In this paper we investigate a special subclass
of such SDSs which are obtained from simpler SDSs
having only 3 base blocks. Their parameter sets have the form
$(v;k_1,k_2,k_3;\lambda)$ with $v\equiv 3 \pmod{4}$
a prime, and have the order $n=(3v-1)/4$.
We assume that they are normalized in the sense that
$v/2>k_1\ge k_2\ge k_3$.
We denote by $\cP$ this infinite family of parameter sets
(see Proposition \ref{beskonacna}).
If we add the Paley-Todd diference set in $\bZ_v$, we
obtain an SDS with four base blocks, whose order $n_0$
is equal to $v$.
Consequently, the Hadamard matrix constructed by using this
new SDS is a skew-Hadamard matrix.

We have conjectered that, for each parameter set in $\cP$,
there exists at least one SDS.

To provide evidence to this conjecture, we performed
many computations to construct such SDSs.
Table 1 gives a summary of the previously known and the
new SDSs belonging to $\cP$ in the range $v\le131$.
Among the 36 parameter sets in this range, the SDSs are
not known in 8 cases.

Moreover, we have constructed three non-equivalent SDSs for the
parameter set \\
$(239;119,112,106;158)$, which belongs to $\cP$
but has $v=239>131$. Each of them gives a
skew-Hadamard matrix of order $956$. These are the first
examples of such matrices.

As an aside, we also record six SDSs with four base blocks
which give skew-Hadamard matrices of order $1324$. Again,
these are the first examples of such matrices.

\section{Acknowledgements}
We thank the referees for their valuable comments and
suggestions.
The authors wish to acknowledge generous support by NSERC.
Computations were performed on the SOSCIP Consortium's Blue Gene/Q, computing platform.  SOSCIP is funded by the Federal Economic Development Agency of Southern Ontario, IBM Canada Ltd., Ontario Centres of Excellence, Mitacs and 14 academic member institutions.

\section{Appendix: SDSs for parameter sets in $\cP$ with
$v\le 131$
}

\noindent $(3;1,1,0;0),  \quad  n = 2$ \\
$\{0\}, \{0\}, \emptyset$ \\

\noindent $(7;3,3,1;2),  \quad  n = 5$ \\
$\{0,1,3\}, \{0,1,3\}, \{0\}$ \\

\noindent $(7;2,2,2;1),  \quad  n = 5$ \\
$\{0,1\}, \{0,2\}, \{0,3\}$ \\

\noindent $(11;4,4,3;3), \quad   n = 8$ \\
$\{0,1,3,5\}, \{0,1,4,5\}, \{0,2,5\}$ \\

\noindent $(19;9,7,6;8),  \quad  n = 14$ \\
$\{0,1,2,3,5,7,12,13,16\}, \{0,1,2,4,5,10,13\}, \{0,1,4,6,8,13\}$ \\

\noindent $(19;7,7,7;7),  \quad  n = 14$ \\
$\{0,1,3,4,7,8,13\}, \{0,1,2,5,8,10,13\}, \{0,1,2,5,7,9,11\}$ \\

\noindent $(23;11,10,7;11),  \quad  n = 17$ \\
$\{0,1,2,3,5,7,8,11,12,15,17\}, \{0,1,2,3,6,8,10,11,14,18\}, \{0,1,2,5,7,11,14\}$ \\

\noindent $(31;15,15,10;17),  \quad   n = 23$ \\
$\{0,1,3,4,6,7,8,9,13,15,17,18,19,24,27\},
\{0,1,2,3,4,5,8,10,11,14,16,17,21,23,25\}$, \\
$\{0,1,4,5,6,9,11,16,19,23\}$ \\

\noindent $(31;13,12,12;14),  \quad   n = 23$ \\
$\{0,1,2,3,5,6,8,13,15,16,19,22,25\},
\{0,1,2,5,7,9,13,14,15,17,22,27\}$, \\
$\{0,1,3,4,5,9,10,11,13,16,20,24\}$ \\

\noindent $(43;21,21,15;25), \quad    n = 32$ \\
$\{0,1,2,4,8,9,10,11,12,14,15,16,19,21,24,27,28,30,32,33,37\}$, \\
$\{0,1,3,4,5,7,8,9,12,13,16,17,19,22,24,26,28,33,34,36,37\}$, \\
$\{0,1,2,4,6,7,10,12,13,15,20,26,27,28,33\}$ \\

\noindent $(43;21,18,16;23), \quad    n = 32$ \\
$\{0,1,2,4,8,9,10,11,12,14,15,16,19,21,24,27,28,30,32,33,37\}$, \\
$\{0,1,2,5,6,7,10,12,14,16,17,20,21,23,24,30,31,32\}$, \\
$\{0,1,2,5,6,8,10,14,17,20,22,23,25,27,33,36\}$ \\

\noindent $(43;20,17,17;22),  \quad   n = 32$ \\
$\{0,2,4,5,6,8,9,11,12,14,17,19,21,22,25,26,31,32,33,35\}$, \\
$\{0,1,2,4,7,9,11,12,15,16,19,20,24,25,26,32,34\}$, \\
$\{0,1,2,3,4,7,8,9,14,17,19,20,23,25,29,32,36\}$ \\

\noindent $(43;19,19,16;22),  \quad   n = 32$ \\
$\{0,1,2,3,4,6,8,11,12,14,16,19,24,25,29,31,34,35,40\}$, \\
$\{0,1,2,3,4,5,7,10,11,13,15,18,19,20,24,25,27,31,32\}$, \\
$\{0,1,2,3,7,9,10,12,15,28,29,32,36\}$ \\

\noindent $(47;22,22,17;26),  \quad  n = 35$ \\
$\{0,1,2,3,4,6,7,8,9,11,16,17,18,19,22,25,27,31,36,37,39,43\}$, \\
$\{0,1,3,4,5,6,7,8,10,14,15,18,19,22,25,27,28,30,31,34,38,40\}$, \\
$\{0,1,2,5,6,7,11,13,15,18,22,24,29,32,34,39,40\}$ \\
Two other solutions were constructed in
\cite[Proposition 2.1]{Djokovic:skewHM:188:388:2008}. \\

\noindent $(47;21,19,19;24),  \quad n = 35$ \\
$\{0,1,2,3,4,6,7,10,11,12,14,15,17,21,23,24,30,31,36,39,41\}$, \\
$\{0,1,3,4,6,7,12,14,16,17,21,22,26,28,29,33,35,37,41\}$, \\
$\{0,1,2,3,4,5,7,9,15,18,19,20,23,25,26,30,34,39,42\}$ \\
Two other solutions were constructed in
\cite[Proposition 2.2]{Djokovic:skewHM:188:388:2008}. \\

\noindent $(59;29,28,22;35),  \quad  n = 44$  \\
\noindent Three non-equivalent SDSs with these parameters can be
constructed from the D-optimal SDS with parameters
$(59;28,22;21)$ found in \cite{FKS:2004} and the two additional D-optimal SDSs found in \cite{DK:D-optimal:2015}.
We have constructed very recently yet another D-optimal SDS with
the same parameters, not equivalent to any of the three SDSs
mentioned above. Its two base blocks are
\begin{eqnarray*}
X_2 &=& \{ 0,2,3,4,5, 6,7,8,9,13,  15,17,19,20,23,
           25,28,29,31,32, 33,36,40,43,44, 45,50,54 \}, \\
X_3 &=& \{ 0,1,2,4,7, 8,9,14,15,17, 18,21,24,26,28,
           33,34,36,39,44, 45,54  \}.
\end{eqnarray*}  \\

In the following two cases we use the subgroup $H=\{1,29,37\}$ of
$\bZ_{67}^*$. The base blocks $X_1,X_2,X_3$ of the SDS are given by the formulas (\ref{orbite}) with $J,K,L$ given below. \\

\noindent $(67;31,28,28;37),  \quad  n = 50$  \\
\begin{eqnarray*}
J &=& \{0, 1,4,5,6,8, 10,16,18,23,27\}, \\
K &=& \{0, 1,2,3,4,5, 12,15,32,34\},  \\
L &=& \{0, 2,3,4,5,8, 10,18,23,30\}.
\end{eqnarray*}

\noindent $(67;30,30,27;37),  \quad  n = 50$  \\
\begin{eqnarray*}
J &=& \{8,12,15,16,17, 25,27,32,34,41\}, \\
K &=& \{1,3,9,12,15,   23,25,32,34,36\}, \\
L &=& \{1,4,6,8,10,    18,23,25,34\}.
\end{eqnarray*}

\noindent $(71;34,32,28;41),  \quad  n = 53$ \quad ? \\

\noindent $(71;31,31,30;39),  \quad  n = 53$ \\
\noindent $H=\{1,5,25,54,57\}$ is the subgroup of order 5 of the group $\bZ_{71}^*$. The base blocks $X_1,X_2,X_3$ of the SDS are given by the formulas (\ref{orbite}) with
\begin{eqnarray*}
J &=& \{0, 1,6,7,11,14,27\}, \\
K &=& \{0, 1,6,9,11,13,27\},  \\
L &=& \{1,2,6,13,14,42\}.
\end{eqnarray*}

\noindent $(79;39,37,31;48),  \quad  n = 59$ \\
\noindent Use the D-optimal SDS with parameters $(79;37,31;29)$ constructed in \cite{Djokovic:AustralasJC:1997}. \\

\noindent $(79;38,35,32;46),  \quad  n = 59$ \quad ? \\

\noindent $(79;37,34,33;45),  \quad  n = 59$ \\
\noindent $H=\{1,23,55\}$ is the subgroup of order 3 of the group $\bZ_{79}^*$. The base blocks $X_1,X_2,X_3$ of the SDS are given by the formulas (\ref{orbite}) with
\begin{eqnarray*}
J &=& \{0, 2,5,6,9,18,   20,22,27,30,33, 34,40\}, \\
K &=& \{0, 1,6,15,18,22, 24,30,33,34,41, 47\},  \\
L &=& \{1,2,3,4,5,    10,11,20,27,33, 41\}. \\
\end{eqnarray*}

\noindent $(83;39,37,34;48),  \quad  n = 62$ \quad ? \\

\noindent $(83;37,37,35;47),  \quad  n = 62$ \quad ? \\

\noindent $(103;51,48,42;64),  \quad  n = 77$ \\
\noindent Use the D-optimal SDSs with parameters
$(103;48,42;39)$ constructed in \cite{DK:JCD:2012}. \\

\noindent $(103;51,46,43;63),  \quad  n = 77$ \\
\noindent Use the D-optimal SDSs with parameters
$(103;46,43;38)$ constructed in \cite{DK:JCD:2012}. \\

In the following two cases we use the subgroup $H=\{1,46,56\}$ of
$\bZ_{103}^*$. The base blocks $X_1,X_2,X_3$ of the SDS are given by the formulas (\ref{orbite}) with $J,K,L$ given below. \\

\noindent $(103;49,49,42;63),  \quad  n = 77$
\begin{eqnarray*}
J &=& \{0, 1,3,4,6,7,10,12,15,17,20,23,29,38,40,42,44\}, \\
K &=& \{0, 2,4,5,6,7,10,12,14,17,19,21,31,40,49,53,55\},  \\
L &=& \{1,3,7,8,12,19,20,21,38,44,47,49,53,60\}.
\end{eqnarray*}

\noindent $(103;46,46,45;60),  \quad  n = 77$
\begin{eqnarray*}
J &=& \{0, 3,5,6,7,10,11,14,17,22,29,31,33,40,49,55\}, \\
K &=& \{0, 1,2,7,10,14,15,17,19,20,29,33,44,47,49,55\}, \\
L &=& \{1,4,5,8,10,12,15,20,22,29,33,49,51,53,62\}.
\end{eqnarray*}

\noindent $(107;49,48,46;63),  \quad  n = 80$ ? \\

\noindent $(127;61,58,54;78),  \quad  n = 95$ ? \\

\noindent $(127;60,57,56;77),  \quad  n = 95$ ? \\

\noindent $(127;57,57,57;76),  \quad  n = 95$ \\
\noindent An SDS with these parameters has been constructed in
\cite{Djokovic:OperMatr:2009}. \\

For $v=131$ there are three parameter sets in $\cP$ as
listed below. We have constructed an SDS for the first
and the third parameter set by using the subgroup
$H=\{1,53,58,61,89\}$ of $\bZ_{131}^*$.
The base blocks $X_1,X_2,X_3$ are given by the
formulas (\ref{orbite}) with $J,K,L$ specified below. \\

\noindent $(131;65,61,55;83),  \quad  n = 98$ \\
\begin{eqnarray*}
J &=& \{ 2,3,4,7,8, 9,11,18,19,33, 36,38,51 \}, \\
K &=& \{0, 1,2,6,7,8, 9,14,22,27,36, 42,44 \}, \\
L &=& \{ 1,2,14,17,21, 22,29,38,42,51, 79 \}.
\end{eqnarray*}

\noindent $(131;64,58,57;81),  \quad  n = 98$ ? \\

\noindent $(131;61,61,56;80),  \quad  n = 98$ \\
\begin{eqnarray*}
J &=& \{ 0, 2,8,11,12,14, 18,19,23,27,29, 33,79 \}, \\
K &=& \{ 0, 2,3,4,8,9, 14,17,27,33,42, 44,79 \}, \\
L &=& \{ 0, 2,4,7,8,14, 17,19,21,27,29, 79 \}.
\end{eqnarray*}

\end{document}